\documentclass[12pt]{amsart}
\usepackage{amssymb,amsmath,tabularx}
\usepackage{amsthm}
\usepackage{exscale}
\usepackage{graphicx,booktabs}
\usepackage[all]{xy}
\usepackage{pstricks,pst-node,pst-plot}
\usepackage[bookmarks=true]{hyperref}
\numberwithin{equation}{section}
\newtheorem{thm}[equation]{Theorem}
\newtheorem*{thankyou}{Acknowledgement}
\newtheorem{cor}[equation]{Corollary}

\newtheorem{lem}[equation]{Lemma}
\newtheorem{conj}[equation]{Conjecture}
\newtheorem{prop}[equation]{Proposition}
\theoremstyle{definition}

\newtheorem{rem}[equation]{Remark}

\newcommand{\id}{\operatorname{id}}

\newcommand{\sgn}{\operatorname{sgn}}

\newcommand{\res}{\operatorname{res}}

\newcommand{\Ext}{\operatorname{\mathrm{Ext}}}

\renewcommand{\mod}{\operatorname{mod\,}}

\begin{document}
\title[Complexity of Specht moules]{The complexity of the Specht modules corresponding to hook
partitions}
\author{Kay Jin Lim}

\maketitle
\begin{abstract} We show that the complexity of the Specht module
corresponding to any hook partition is the $p$-weight of the
partition. We calculate the variety and the complexity of the signed
permutation modules. Let $E_s$ be a representative of the conjugacy class containing an elementary
abelian $p$-subgroup of a symmetric group generated by $s$ disjoint $p$-cycles. We give formulae for
the generic Jordan types of signed permutation modules
restricted to $E_s$ and of Specht modules corresponding to hook
partitions $\mu$ restricted to $E_s$ where
$s$ is the $p$-weight of $\mu$.
\end{abstract}

\section{Introduction}

Alperin and Evens \cite{JALE} introduced the complexity of finitely
generated modules over finite group algebras. Meanwhile, Carlson
\cite{JC} had been studying some varieties for finitely generated
modules over finite group algebras. Over elementary abelian
$p$-groups, he showed that the complexity of a module is precisely
the dimension of the cohomological variety and the dimension of the
rank variety corresponding to the module. Avrunin and Scott
\cite{GALS} proved an analogue of Quillen's Stratification. The
complexity of a $kG$-module can be determined by looking at the
restriction of the module to the elementary abelian $p$-subgroups of
the group $G$. The study of the varieties for modules over
elementary abelian $p$-groups is closely related to the study of the
notion of generic Jordan types of the modules introduced by
Wheeler \cite{WW}.

A partition $\mu=({\mu_1}^{a_1}, \ldots, {\mu_s}^{a_s})$ is $p\times
p$ if for each $1\leq i\leq s$, both $\mu_i$ and $a_i$ are multiples of
$p$. The VIGRE group in Georgia made the following conjecture.

\begin{conj}[VIGRE 2004]\label{conjecture VIGRE two}
The complexity of the Specht module $S^\mu$ is the $p$-weight of the
partition $\mu$ if and only if $\mu$ is not $p\times p$.
\end{conj}

In \cite{KJL}, we studied the support varieties and the complexities
for Specht modules corresponding to some $p$-regular partitions and
the partition $(p^p)$. We showed that, for the case of abelian
defects, Conjecture \ref{conjecture VIGRE two} is true. We also
showed that a large class of Specht modules satisfy Conjecture
\ref{conjecture VIGRE two}. Recently \cite{DH2}, Hemmer shows one
direction of Conjecture \ref{conjecture VIGRE two}, i.e., if a
partition $\mu$ is $p\times p$, then the complexity of the Specht
module $S^\mu$ is strictly less than the $p$-weight of $\mu$.

Let $E$ be an elementary abelian $p$-group of rank $m$. To show that
a $kE$-module $M$ has complexity $m$, it suffices to show that the
module $M$ is not generically free (see Proposition \ref{not
generically free equivalent to maximal complexity}). Let $\mu$ be a
hook partition of $n$, $s$ be the $p$-weight of $\mu$ and $E_s$ be
the elementary abelian $p$-subgroup of $\mathfrak{S}_n$ generated by
the $p$-cycles $((i-1)p+1,(i-1)p+2,\ldots, ip)$ with $1\leq i\leq
s$. We consider the restricted module $S^\mu{\downarrow_{E_s}}$ and
show that the module is not generically free. This implies that the
complexity of the Specht module $S^\mu$ is bounded below by $s$.
Since the value $s$ is also the $p$-rank of a defect group of the
block containing $S^\mu$, we get the other half of Conjecture
\ref{conjecture VIGRE two} for hook partitions.

\begin{rem} Following 4.1 of \cite{KJL}, one may wonder if $V_{\mathfrak{S}_n}(S^\mu)=\res^\ast_{\mathfrak{S}_,D_{\widetilde{\mu}}}V_{D_{\widetilde{\mu}}}(k)$ whenever $\mu$ is not $p\times p$. A counterexample to the above statement is the partition $\mu=(7,1^3)$ where the vertex of $S^\mu$ is $C_3\times C_3\times C_3$ \cite{MW} and $D_{\widetilde{\mu}}\cong C_3\wr C_3$.
\end{rem}

\begin{thm}\label{complexity corresponding to hook partitions} For any hook partition $\mu$, the
complexity of the Specht module $S^\mu$ is exactly the $p$-weight of
$\mu$.
\end{thm}

\section{Background materials and notations}

Most of the basic materials about group cohomology and the
representation theory of symmetric groups can be found in \cite{DB}
and \cite{GJAK} respectively.

Let $G$ be a finite group, $k$ be an algebraically closed field of
characteristic prime $p$ and $\Ext^\ast_{kG}(M,M)$ be the cohomology
ring of a finitely generated $kG$-module $M$. The cohomological
variety $V_G(M)$ for the module $M$ is the set of maximal ideals spectrum of
$\Ext^{\bullet}_{kG}(k,k)$ containing the kernel of the map
$\Phi_M:\Ext^{\bullet}_{kG}(k,k)\xrightarrow{\otimes_k M}
\Ext^\ast_{kG}(M,M)$ where $$\Ext^{\bullet}_{kG}(k,k)=\left
\{\begin{array}{ll} \Ext^{\text{ev}}_{kG}(k,k)& \text{$p$ is odd}\\
\Ext^\ast_{kG}(k,k)& p=2.\end{array}\right .$$ We have Theorem \cite{GALS}
$$V_G(M)=\bigcup_{E\in\mathcal{E}(G)}\res^\ast_{G,E}V_E(M)$$ where
$\mathcal{E}(G)$ is a set of representatives for the conjugacy
classes of elementary abelian $p$-subgroups of $G$ and
$\res^\ast_{G,E}:V_E(k)\to V_G(k)$ is the map induced by the
restriction $\res_{G,E}:\Ext^\bullet_{kG}(k,k)\to
\Ext^\bullet_{kE}(k,k)$. So $\dim
V_G(M)=\max_{E\in\mathcal{E}(G)}\{\dim V_E(M)\}$. Let $E$ be an
elementary abelian $p$-group of rank $n$ generated by the elements
$g_1,g_2,\ldots,g_n$. The rank variety $V_E^\sharp(M)$ of a finitely
generated $kE$-module $M$ is the set
$$\{0\neq \alpha\in k^n\,|\,\text{$M{\downarrow_{k\langle
u_\alpha\rangle}}$ is not free}\}\cup \{0\}$$ where
$u_\alpha=1+\sum^n_{i=1}\alpha_i(g_i-1)$ and
$\alpha=(\alpha_1,\ldots,\alpha_n)\in k^n$. Avrunin and Scott
\cite{GALS} showed that $V_E^\sharp(M)\cong V_E(M)$.

Let $\alpha$ be a generic point in $k^n$. The generic Jordan type of a
finitely generated $kE$-module $M$ is the Jordan type of the
restricted module $M{\downarrow_{\langle u_\alpha\rangle}}$ where
$\langle u_\alpha\rangle $ is the cyclic group $C_p$ of order $p$
described above \cite{WW}. In this case, we write $(1^{n_1},\ldots,p^{n_p})$ for the generic Jordan type of $M$ if the number of Jordan type of size $1\leq i\leq p$ is $n_i$. Suppose that $\beta$ is a multiple of $\alpha$. We
have $M{\downarrow_{\langle u_\beta\rangle}}\cong
M{\downarrow_{\langle u_\alpha\rangle}}$. We write
$[\alpha]^\ast(M)$ for the isomorphism class of $kC_p$-modules
containing $M{\downarrow_{\langle u_\alpha\rangle}}$. We now state a
partial result of 4.7 \cite{EFJPAS} for our case.

\begin{prop}[Part of 4.7 \cite{EFJPAS}]
Let $E$ be an elementary abelian $p$-group, $M$ and $N$ be finitely
generated $kE$-modules and $\alpha\in k^n$ be a generic point. We
have $[\alpha]^\ast(M\oplus N)\cong [\alpha]^\ast(M)\oplus
[\alpha]^\ast(N)$.
\end{prop}

The stable generic Jordan type of $M$ is the generic Jordan type of
$M$ modulo its projective summands. With the notation introduced earlier, the stable generic Jordan type of $M$ is $(1^{n_1},\ldots, (p-1)^{n_{p-1}})$.

\begin{prop}\label{not generically free equivalent to maximal
complexity} Let $E$ be an elementary abelian $p$-group of rank $n$.
A $kE$-module $M$ is not generically free if and only if
$V^\sharp_E(M)=V^\sharp_E(k)$.
\end{prop}
\begin{proof} Suppose that $V^\sharp_E(M)$ has dimension $m<n$. Let
$\alpha\in k^n$ be a generic point and $l_\alpha\subseteq k^n$ be
the line containing the point $\alpha$. Note that $V^\sharp_E(M)$ is
a closed homogeneous affine variety (see Theorem 4.3 of \cite{JC1}).
Since $m+1\leq n$, we have $V_E^\sharp(M)\cap l_\alpha=\{0\}$, i.e.,
$M$ is generically free.
\end{proof}

Let $n$ be a natural number. A partition $\mu$ of $n$ is a
non-increasing sequence of positive integers $(\mu_1,\ldots,\mu_s)$
such that $\sum^s_{i=1}\mu_i=n$. In this case, we write $n=|\mu|$.
The Young diagram $[\mu]$ is the subset of $\mathbb{Z}^2$ consisting of
all nodes $(i,j)$ satisfying $1\leq i\leq s$ and $1\leq j\leq
\mu_i$. The $p$-core $\widetilde{\mu}$ of $\mu$ is the partition
corresponding to the Young diagram obtained from $[\mu]$ by removing as
many skew $p$-hooks as possible. The number of skew $p$-hooks
removed from $[\mu]$ to get $[\widetilde{\mu}]$ is called the
$p$-weight of $\mu$. A $\mu$-tableau is an assignment of the numbers
$1,2,\ldots, n$ to the nodes of the Young diagram $[\mu]$. For each
$(i,j)\in[\mu]$, we write $t_{ij}$ for the number assigned to
$(i,j)$. For each number $1\leq m\leq n$, we write $t_m$ for the
node $(i,j)$ such that $t_{ij}=m$. We write $R_i(t)$ for the set
consisting of the numbers in the $i$th row of $t$, i.e.,
$R_i(t)=\{t_{ij}\,|\,j\geq 1\}$. Similarly, we write $C_j(t)$ for
the set consisting of the numbers in the $j$th column of $t$. If
numbers are increasing down each column and along each row of $t$,
we say that $t$ is a standard $\mu$-tableau.

Let $\mathfrak{S}_n$ denote the symmetric group on $n$ letters. The
group acts on the set of all $\mu$-tableaux $t$ by permuting the
numbers assigned to $t$. Let $R_t$ and $C_t$ be the row stabilizer
and column stabilizer of $t$ respectively, i.e.,
$R_t=\mathfrak{S}_{R_1(t)}\times \ldots\times \mathfrak{S}_{R_s(t)}$
and $C_t=\mathfrak{S}_{C_1(t)}\times \ldots \times
\mathfrak{S}_{C_{\mu_1}(t)}$ where $\mathfrak{S}_\Omega$ is the
symmetric group corresponding to a given set $\Omega$. We define an
equivalence relation on the set of all $\mu$-tableaux, $s\sim t$ if
and only if $s=\sigma t$ for some $\sigma\in R_t$. A $\mu$-tabloid
$\{t\}$ is the equivalence class containing the $\mu$-tableau $t$.
The $\mu$-polytabloid corresponding to the $\mu$-tableau $t$ is
$$e_t=\sum_{\sigma\in C_t}\sgn(\sigma)\{\sigma t\}.$$ The $k$-vector space spanned by all
$\mu$-polytabloids forms a $k\mathfrak{S}_n$-module. It is called
the Specht module $S^\mu$. If $t$ is a standard $\mu$-tableau, the
$\mu$-polytabloid $e_t$ is called a standard $\mu$-polytabloid. The
set of all standard $\mu$-polytabloids forms a basis of $S^\mu$, the
standard basis of $S^\mu$. The Young subgroup $\mathfrak{S}_\mu$ of
$\mathfrak{S}_n$ corresponding to $\mu$ is
$$\mathfrak{S}_{\{1,\ldots,\mu_1\}}\times
\mathfrak{S}_{\{\mu_1+1,\ldots, \mu_1+\mu_2\}}\times \ldots \times
\mathfrak{S}_{\{\mu_1+\ldots+\mu_{s-1}+1,\ldots,
\mu_1+\ldots+\mu_s\}}.$$ The signed permutation module corresponding
to a pair of partitions $(\alpha,\beta)$ with $|\alpha|=a$ and $|\beta|=b$ is the induced module
$$M(\alpha|\beta)=(k\boxtimes
\sgn){\uparrow^{\mathfrak{S}_{a+b}}_{\mathfrak{S}_\alpha\times
\mathfrak{S}_\beta}}$$ where $\boxtimes$ denotes the exterior tensor
product of two modules \cite{SD2}. It generalizes the notion of
permutation modules, by taking $b=0$.

For each $1\leq s\leq \lfloor n/p\rfloor$, we write $E_s$ for the
elementary abelian $p$-subgroup of $\mathfrak{S}_n$ generated by the
$p$-cycles $g_i=((i-1)p+1, (i-1)p+2, \ldots, ip)$ with $1\leq i\leq
s$. For each positive integer $i$, we write $I_i$ for the set
$\{(i-1)p+1, (i-1)p+2, \ldots, ip\}$.

\begin{prop} Let $\mu$ be a partition of $n=dp+r$ with $0\leq r\leq
p-1$.
\begin{enumerate}
\item [(i)]\textup{[The Branching Theorem \S9 of \cite{GJ}]} Let $\Omega(\mu)$
be the set of partitions of $n-1$ obtained from $\mu$ by removing a
node. The module $S^\mu{\downarrow_{\mathfrak{S}_{n-1}}}$ has a
Specht filtration with Specht factors $S^\lambda$ one for each
$\lambda\in\Omega(\mu)$.

\item [(ii)]\textup{[Nakayama's Conjecture]} Let $\lambda$ be
another partition of $n$. The Specht modules $S^\mu, S^\lambda$ lie
in the same block if and only if the corresponding partitions
$\mu,\lambda$ have the same $p$-cores.

\item [(iii)] If $|\widetilde{\mu}|>r$, then
$S^\mu{\downarrow_{E_d}}$ is generically free.
\end{enumerate}
\end{prop}
\begin{proof} For a proof of Nakayama's Conjecture, see \S 6 of
\cite{GJAK}. The proof of (iii) is similar to the proof of
Proposition 2.2 (iv) \cite{KJL}. Since the $p$-weight $m_\mu$ of
$\mu$ is strictly less than $d$, a defect group $D_\mu$ of the block
containing $S^\mu$ has $p$-rank $m_\mu$ and
$V_{\mathfrak{S}_n}(S^\mu)\subseteq \res^\ast_{\mathfrak{S}_n,
D_\mu} V_{D_\mu}(S^\mu)$ (Proposition 2.1 (iv) of \cite{KJL}), we
have $\dim V_{E_d}(S^\mu)\leq m_\mu <d$. So
$S^\mu{\downarrow_{E_d}}$ is generically free by Proposition \ref{not generically free equivalent to maximal complexity}.
\end{proof}

\section{Signed permutation modules}\label{signed permutation modules}

Let $n_1, n_2,\ldots, n_u$ and $s$ be non-negative integers. We define the
set $\Lambda(n_1,\ldots,n_u;s)$ to consist of all $u$-tuples
$(c_1,\ldots,c_u)\in(\mathbb{N}_{\geq 0})^u$ such that $0\leq
c_i\leq n_i$ for each $1\leq i\leq u$ and $c_1+c_2+\ldots+c_u=s$. If
$s>\sum^u_{i=1}n_i$ or $s<0$, then
$\Lambda(n_1,\ldots,n_u;s)=\varnothing$.

\begin{thm}\label{sgjt for signed permutation module} Let $\alpha=(\alpha_1,\ldots, \alpha_m)$,
$\beta=(\beta_1,\ldots,\beta_n)$ and $|\alpha|+|\beta|=dp+r$ with
$0\leq r\leq p-1$. Suppose that the $p$-residue of $\alpha_i$ is
$s_i$ for each $1\leq i\leq m$, the $p$-residue of $\beta_j$ is
$s_{m+j}$ for each $1\leq j\leq n$ and $\sum^{m+n}_{i=1}s_i=cp+r$.
Let $1\leq s\leq d$.

\begin{enumerate}
\item [(i)] The stable generic Jordan type of
$M(\alpha|\beta){\downarrow_{E_s}}$ is $\left (1^{N(\alpha,\beta,s)}\right )$ where
$$N(\alpha,\beta,s)=\sum_{(c_1,\ldots,c_{m+n})\in
\Lambda}\left (\frac{s!}{\prod_{i=1}^{m+n}c_i!} \cdot
\frac{((d-s)p+r)!}{\prod_{i=1}^{m}(\alpha_i-c_ip)! \prod_{j=1}^n(\beta_j-c_{j+m}p)!}\right )
$$ where $\Lambda=\Lambda\left (\frac{\alpha_1-s_1}{p}, \ldots,
\frac{\alpha_m-s_m}{p},\frac{\beta_1-s_{m+1}}{p},
\ldots,\frac{\beta_n-s_{m+n}}{p};s\right )$. The empty sum is
defined to be zero. In this case, the module is generically free.

\item [(ii)] We have
$V_{\mathfrak{S}_{|\alpha|+|\beta|}}(M(\alpha|\beta))=
\res^\ast_{\mathfrak{S}_{|\alpha|+|\beta|},\mathfrak{S}_\alpha\times
\mathfrak{S}_\beta}V_{\mathfrak{S}_\alpha\times
\mathfrak{S}_\beta}(k)$. In particular, the complexity of the signed
permutation module $M(\alpha|\beta)$ is $d-c$.
\end{enumerate}
\end{thm}
\begin{proof} We use the Mackey decomposition formula (see 3.3.4 \cite{DB}), $$(k\boxtimes
\sgn){\uparrow^{\mathfrak{S}_{|\alpha|+|\beta|}}_{\mathfrak{S}_\alpha\times
\mathfrak{S}_\beta}}{\downarrow_{E_s}}\cong
\bigoplus_{E_sg(\mathfrak{S}_\alpha\times
\mathfrak{S}_\beta)}{}^g(k\boxtimes \sgn){\downarrow_{E_s\cap
{}^g(\mathfrak{S}_\alpha\times
\mathfrak{S}_\beta)}}{\uparrow^{E_s}}.$$ If $E_s\cap
{}^g(\mathfrak{S}_\alpha\times \mathfrak{S}_\beta)\lneq E_s$, then
${}^g(k\boxtimes \sgn){\downarrow_{E_s\cap
{}^g(\mathfrak{S}_\alpha\times
\mathfrak{S}_\beta)}}{\uparrow^{E_s}}$ is generically free. Suppose
that $E_s\cap {}^g(\mathfrak{S}_\alpha\times \mathfrak{S}_\beta)=
E_s$, i.e., $E_sg\subseteq g(\mathfrak{S}_\alpha\times
\mathfrak{S}_\beta)$. The double coset representatives of the
subgroups $E_s, \mathfrak{S}_\alpha\times \mathfrak{S}_\beta$ in
$\mathfrak{S}_{|\alpha|+|\beta|}$ correspond to the orbits of the
$\mu$-tabloids under the action of $E_s$ where
$\mu=(\alpha_1,\ldots,\alpha_m,\beta_1,\ldots,\beta_n)$. So the
number of double coset representatives fixed by $E_s$ is precisely
the number of $\mu$-tabloids fixed by $E_s$. A $\mu$-tabloid $\{t\}$
is fixed by $E_s$ if and only if for each $1\leq i\leq s$, we have
$I_i\subseteq R_{j(i)}(t)$ for some $1\leq j(i)\leq m+n$. In this
case, it is necessary that
$$\sum^m_{i=1}(\alpha_i-s_i )+\sum^n_{j=1}(\beta_j-s_{m+j})\geq sp$$ i.e., $s\leq d-c$. So
$\Lambda=\varnothing$ if and only if $s>d-c$. Suppose that $s\leq
d-c$. We fix an element $(c_1,c_2,\ldots,c_{m+n})$ in the set
$\Lambda$. The number of choices assigning each $I_i$ with $1\leq
i\leq s$ into a row of the partition $(c_1p,c_2p,\ldots,c_{m+n}p)$
is
$$\frac{s!}{\prod_{i=1}^{m+n}c_i!}.$$ Independently, the number of
choices assigning the remaining $(d-s)p+r$ numbers $sp+1,sp+2,
\ldots, dp+r$ into the remaining $(d-s)p+r$ nodes of $[\mu]$ with
$\mu_i-c_ip$ nodes in each $i$th row is
$$\frac{((d-s)p+r)!}{\prod_{i=1}^{m}(\alpha_i-c_ip)! \prod_{j=1}^n(\beta_j-c_{j+m}p)!}.$$
If we sum up over all elements of $\Lambda$, we get
$N(\alpha,\beta,s)$. In
these cases, the generic Jordan type of ${}^g(k\boxtimes
\sgn)_{E_s}$ is $(1)$. This completes the proof for (i).

Let $G=\mathfrak{S}_{|\alpha|+|\beta|}$ and
$H=\mathfrak{S}_\alpha\times \mathfrak{S}_\beta$. By Proposition
8.2.4 of \cite{LE} and Proposition 2.1 (iii) \cite{KJL}, we have
$V_G(M(\alpha|\beta))=\res^\ast_{G,H}V_H(k\boxtimes
\sgn)=\res^\ast_{G,H}V_H(k)$. Since the map $\res^\ast$ is a finite
map (4.2.5 of \cite{DB} II), we have
\begin{align*}
\dim V_G(M(\alpha|\beta))&= \dim V_H(k)\\
&=\text{$p$-rank of $H$}\\
&=\sum^m_{i=1}(\alpha_i-s_i)/p+\sum^n_{j=1}(\beta_j-s_{j+m})/p\\
&=d-c\qedhere
\end{align*}
\end{proof}

\begin{rem}\label{JO remark} Theorem \ref{sgjt for signed permutation module} (ii) is
an obvious generalization of 3.2.2 \cite{DHDN}.
\end{rem}

\section{Proof of Theorem \ref{complexity corresponding to hook
partitions}}

The proof of Theorem \ref{complexity corresponding to hook
partitions} is a consequence of a more general statement, which we
explicitly compute the stable generic Jordan type of
$S^\mu{\downarrow_{E_s}}$ where $s$ is the $p$-weight of a hook
partition $\mu$ (Corollary \ref{corollary to hook not multiple of p}
and Theorem \ref{generic jordan type of hook partition of case p a
multiple}). Let $\mu=(a,1^b)$. Our aim is to show that
$S^\mu{\downarrow_{E_s}}$ is not generically free. We consider two
cases, $p\nmid a+b$ and $p\mid a+b$. We shall briefly describe the
proofs of Theorem \ref{generic jordan type of hook partition of case
p not a multiple} and Theorem \ref{generic jordan type of hook
partition of case p a multiple}.

By the Littlewood-Richardson Rule (see 2.8.13 of \cite{GJAK}), the
signed permutation module $M((a)|(b))$ has a Specht filtration with
Specht factors $S^{(a,1^b)}$ and $S^{(a+1,1^{b-1})}$. In the case
where $a+b\not\equiv 0(\mod p)$, the sizes of $p$-cores are nonzero
and $b\not \equiv b+1(\mod p)$, so $p$-cores of $(a,1^b),
(a+1,1^{b-1})$ are distinct. Using Nakayama's Conjecture, we have a
direct sum decomposition $M((a)|(b))\cong S^{(a,1^b)}\oplus
S^{(a+1,1^{b-1})}$. We prove Theorem \ref{generic jordan type of
hook partition of case p not a multiple} by using induction on $b$.

The case for $a+b\equiv 0(\mod p)$ is slightly more complicated. We
show that the short exact sequence $$0\to S^{(a-1,1^{b+1})}\to
S^{(a,1^{b+1})}\to S^{(a,1^b)}\to 0$$ generically splits, i.e.,
$S^{(a,1^{b+1})}{\downarrow_{\langle u_\alpha\rangle}}\cong
S^\mu{\downarrow_{\langle u_\alpha\rangle}}\oplus
S^\lambda{\downarrow_{\langle u_\alpha\rangle}}$ for a generic point
$\alpha\in k^d$ where $a+b=dp$. With Corollary \ref{corollary to
hook not multiple of p}, we prove Theorem \ref{generic jordan type
of hook partition of case p a multiple} by using induction on $b$.

\subsection{Hook of size not a multiple of $p$}

\begin{thm}\label{generic jordan type of hook partition of case p not a multiple}
Let $\mu=(a,1^b)$, $a+b=dp+r$, $a=up+a_0$ and $b=vp+b_0$ with $0\leq
r,a_0,b_0\leq p-1$ and $r\neq 0$. For any $1\leq s\leq d$, the
stable generic Jordan type of $S^\mu{\downarrow_{E_s}}$ is
$\left (1^{N(\mu;s)}\right )$ where
$$N(\mu;s)=\sum_{(c_1,c_2)\in\Lambda(u,v;s)}\binom{s}{c_2}\binom{(d-s)p+r-1}{b-c_2p}.$$
\end{thm}
\begin{proof} If $d=0$, there is nothing to prove. For any hook $(a,1^b)$, we write
$\Lambda((a,1^b);s)$ for the set $\Lambda(u,v;s)$. We now fix the
numbers $a,b$. Let $\mu=(a,1^b)$ and $\lambda=(a-1,1^{b+1})$. We
prove the formula by using induction on the number $b$. If $b=0$,
then $S^\mu$ is the trivial module and it has Jordan type $(1)$. On
the other hand, the set $\Lambda(d,0;s)$ contains precisely one
element $(s,0)$. So $N(\mu;s)=\binom{s}{0}\binom{(d-s)p+r-1}{0}=1$
given that $r\geq 1$. Suppose that for some $0\leq b$, the module
$S^\mu{\downarrow_{E_s}}$ has stable generic Jordan type as given by
the formula. Since $r\neq 0$, we have a direct sum decomposition
$M((a-1)|(b+1))\cong S^\mu\oplus S^\lambda$. Let $a_1, b_1$ be the
$p$-residues of $a-1, b+1$ respectively. It is clear that $b_1\equiv
b_0+1(\mod p)$ and $a_0\equiv a_1+1(\mod p)$. By Theorem \ref{sgjt for
signed permutation module}, $M((a-1)|(b+1)){\downarrow_{E_s}}$ has
stable generic Jordan type $\left (1^{N((a-1),(b+1),s)}\right )$ where
$$N((a-1),(b+1),s)=\sum_{(c_1,c_2)\in
\Lambda(\lambda;s)}\binom{s}{c_2}\binom{(d-s)p+r}{b+1-c_2p}.$$ Our aim is to show that $N(\lambda;s)=N((a-1),(b+1),s)-N(\mu;s)$. We
consider 4 cases.

\noindent Case (i): Suppose that $r\leq b_0\leq b_1$. If $b_0>r$, then
$\Lambda(\lambda;s)=\Lambda(\mu;s)$ and $c=1$. The stable generic
Jordan type of $S^\lambda{\downarrow_{E_s}}$ is $(1^w)$ where
\begin{align*}
w=& \sum_{(c_1,c_2)\in
\Lambda(\lambda;s)}\binom{s}{c_2}\binom{(d-s)p+r}{b+1-c_2p}-\sum_{(c_1,c_2)\in
\Lambda(\mu;s)}\binom{s}{c_2}\binom{(d-s)p+r-1}{b-c_2p}\\
=&\sum_{(c_1,c_2)\in \Lambda(\lambda;s)}\binom{s}{c_2}\left
(\binom{(d-s)p+r}{b+1-c_2p}-\binom{(d-s)p+r-1}{b-c_2p}\right )\\
=&\sum_{(c_1,c_2)\in
\Lambda(\lambda;s)}\binom{s}{c_2}\binom{(d-s)p+r-1}{b+1-c_2p}=N(\lambda;s)
\end{align*} Suppose that $b_0=r$, we have $\Lambda(\lambda;s)\cup \{(u,s-u)\}=\Lambda(\mu;s)$
if $s<d$; otherwise, $\Lambda(\lambda;s)=\varnothing$. In the first
case, $\binom{(d-s)p+r-1}{b-(s-u)p}=0$ given that
$b-(s-u)p=vp+r-(s-u)p=(d-s)p+r>(d-s)p+r-1$. So the stable generic
Jordan type of $S^\lambda{\downarrow_{E_s}}$ is
$(1^w)=\left (1^{N(\lambda;s)}\right )$. In the second case, the module
$M((a-1)|(b+1)){\downarrow_{E_s}}$ is generically free. As a direct
summand of $M((a-1)|(b+1)){\downarrow_{E_s}}$, the module
$S^\lambda{\downarrow_{E_s}}$ is generically free. This fits into
the formula.

\noindent Case (ii): Suppose that $r\leq b_0$ and $b_1<r$, i.e., $b_1=0$
and $b_0=p-1$. Let $a_0>0$. We have
$\Lambda(\lambda;s)=\Lambda(\mu;s)\cup \{(s-v-1,v+1)\}$ if $s\geq
v+1$; otherwise, $\Lambda(\lambda;s)=\Lambda(\mu;s)$. The second
case is easy. For the first case, we have an extra term in
$N((a-1),(b+1),s)$ which is
$\binom{s}{v+1}\binom{(d-s)p+r}{b+1-(v+1)p}=\binom{s}{v+1}$; on the
other hand, we also have an extra term in $N(\lambda;s)$ which is
$\binom{s}{v+1}\binom{(d-s)p+r-1}{b+1-(v+1)p}=\binom{s}{v+1}$. This
shows the desired formula for the stable generic Jordan type of
$S^\lambda{\downarrow_{E_s}}$. Let $a_0=0$. In this case $r=p-1$. So
we have $\Lambda(\lambda;s)=\left (\Lambda(\mu;s)-\{(u,s-u)\}\right
)\cup \{(s-v-1,v+1)\}$ if $s\geq v+1$; otherwise,
$\Lambda(\lambda;s)=\Lambda(\mu;s)-\{(u,s-u)\}$. For the extra
element $(u,s-u)$, we have
$b-(s-u)p=vp+(p-1)-sp+up=(d-s)p+(p-1)>(d-s)p+r-1$. So the extra term
corresponding to $(u,s-u)$ is superfluous in $N(\mu;s)$. Now the
inductive argument follows similarly as the case where $a_0>0$.

\noindent Case (iii): Suppose that $b_0\leq b_1<r$. In this case,
$\Lambda(\lambda;s)=\Lambda(\mu;s)$. So the inductive step is easy,
the stable generic Jordan type of $S^\lambda{\downarrow_{E_s}}$ is
$\left (1^{N(\lambda;s)}\right )$.

\noindent Case (iv): Suppose that $b_0<r$ and $r\leq b_1$, i.e., $b_0=r-1$
and $b_1=r$. In this case, $\Lambda(\lambda;s)=\Lambda(\mu;s)$. So
the stable generic Jordan type of $S^\lambda{\downarrow_{E_s}}$ is
$\left (1^{N(\lambda;s)}\right )$.
\end{proof}

\begin{cor}\label{corollary to hook not multiple of p}
Let $\mu=(a,1^b)$, $a+b=dp+r$ with $d\neq 0$, $1\leq r\leq p-1$ and $b=vp+b_0$
with $0\leq b_0\leq p-1$.
\begin{enumerate}
\item [(i)] If $r\leq b_0$, then $V_{E_{d-1}}^\sharp(S^\mu)=V_{E_{d-1}}^\sharp(k)$
and the complexity of $S^\mu$ is $d-1$. In this case, the stable
generic Jordan type of $S^\mu{\downarrow_{E_{d-1}}}$ is $\left (1^{N(\mu;
d-1)}\right )$ with $$N(\mu;d-1)=\binom{d-1}{v}\binom{p+r-1}{b_0}\neq
0.$$
\item [(ii)] If $b_0<r$, then $V_{E_d}^\sharp(S^\mu)=V^\sharp_{E_d}(k)$
and the complexity of $S^\mu$ is $d$. In this case, the stable
generic Jordan type of $S^\mu{\downarrow_{E_d}}$ is $\left (1^{N(\mu;d)}\right )$
with $$N(\mu;d)=\binom{d}{v}\binom{r-1}{b_0}\neq 0.$$
\end{enumerate}
\end{cor}
\begin{proof} Suppose that $d\geq 1$; otherwise, the result is trivial.
In general, the complexity of an indecomposable $kG$-module is
bounded above by the $p$-rank of a defect group of the block
containing the module (see 2.1 (iv) of \cite{KJL}). For $b_0\geq r$, a defect group of the block
containing $S^\mu$ is the Sylow $p$-subgroup of the symmetric group
$\mathfrak{S}_{(d-1)p}$, it has $p$-rank $d-1$. Consider
$S^\mu{\downarrow_{E_{d-1}}}$ and apply the formula in Theorem
\ref{generic jordan type of hook partition of case p not a
multiple}, we have
\begin{align*}
N(\mu;d-1)&=\sum_{(c_1,c_2)\in\Lambda(u,v;d-1)}\binom{d-1}{c_2}\binom{p+r-1}{b-c_2p}\\
&=\binom{d-1}{v}\binom{p+r-1}{b_0}.
\end{align*} Note that $N(\mu;d-1)\neq 0$ unless $v=d$ and $b_0=r$,
i.e., $a=0$. For $b_0<r$, a defect group of the block containing
$S^\mu$ is the Sylow $p$-subgroup of the symmetric group
$\mathfrak{S}_{dp}$. Apply Theorem \ref{generic jordan type of hook
partition of case p not a multiple} with $s=d$, we have
$$N(\mu;d)=\sum_{(c_1,c_2)\in\Lambda(u,v;d)}\binom{d}{c_2}\binom{r-1}{b-c_2p}=\binom{d}{v}\binom
{r-1}{b_0}\neq 0.\qedhere$$
\end{proof}

\subsection{Hook of size a multiple of $p$}\label{multiple of p
section}

Let $\mu, \lambda$ be partitions of $n$. We write $\lambda\trianglelefteq \mu$
if $\mu$ dominates $\lambda$ by the total ordering. In the case
where $\lambda\trianglelefteq\mu$ and $\lambda\neq \mu$, we write $\lambda\vartriangleleft
\mu$. Let $\mu$ be a hook partition with $|\mu|\geq p$ and $t$ be a
$\mu$-tableau. We associate $t$ to a partition
$\lambda(t)=(u,1^{p-u})$ of $p$ where $u=|R_1(t)\cap I_1|$.

\begin{lem}\label{subhooks occur} Let $\mu$ be a hook partition
with $|\mu|\geq p$ and $t$ be a standard $\mu$-tableau.
\begin{enumerate}
\item [(i)] The standard
$\mu$-polytabloids $e_s$ involved in $g_1e_t$ satisfy the ordering
$\lambda(s)\trianglelefteq\lambda(t)$. In the case $\lambda(s)=\lambda(t)$, we
have $R_1(s)-I_1=R_1(t)-I_1$. In the case $\lambda(s)\vartriangleleft \lambda(t)$,
we have $(R_1(t)-I_1)\cup \{m\}=R_1(s)-I_1$ for some number $m\in
C_1(t)$.

\item [(ii)] For $2\leq i\leq \lfloor |\mu|/p\rfloor$, $g_i$ permutes the standard $\mu$-polytabloids up to a
sign, $R_1(g_it)-I_i=R_1(t)-I_i$ and $\lambda(g_it)=\lambda(t)$.
Furthermore, $g_ie_t=\pm e_t$ if and only if $I_i\subseteq R_1(t)$
or $I_i\subseteq C_1(t)$. In this case, $g_ie_t=e_t$.
\end{enumerate}
\end{lem}
\begin{proof} All we need are the Garnir relations. Consider
$g_1e_t=e_{g_1t}$. Note that $e_{g_1t}=\varepsilon e_w$ where
$\varepsilon=\pm 1$ and $w$ is the $\mu$-tableau obtained from
$g_1t$ by first rearranging numbers in the first row of $g_1t$ except $2$
and then numbers in the first column of $g_1t$ such that numbers are
increasing along the first row ignoring the first node and numbers
are increasing down the first column. Note that
$R_1(w)-I_1=R_1(t)-I_1$. If $1\in C_1(w)$, then $w$ is standard. In
this case, $\lambda(w)=\lambda(t)$ and $R_1(w)-I_1=R_1(t)-I_1$.
Suppose that $1\not\in C_1(w)$, i.e., $w_{12}=1$. We use the Garnir
relation for the first two columns of $w$. Consider the left coset
representatives $(1,u)$ with $u\in C_1(w)\cup \{1\}$ of
$\mathfrak{S}_{C_1(w)}\times \mathfrak{S}_{\{1\}}$ in
$\mathfrak{S}_{C_1(w)\cup \{1\}}$. So $$g_1e_t=\varepsilon\sum_{m\in
C_1(w)}e_{(1m)w}.$$ Note that for each $m\in C_1(w)$, $e_{(1m)w}$ is
standard up to a sign. If $m\in C_1(w)\cap I_1$, then
$\lambda((1m)w)=\lambda(t)$ and
$R_1((1m)w)-I_1=R_1(w)-I_1=R_1(t)-I_1$. If $m\in C_1(w)-I_1\subseteq
C_1(t)$, then $\lambda((1m)w)\vartriangleleft \lambda(t)$ and
$R_1((1m)w)-I_1=(R_1(w)-I_1)\cup \{m\}=(R_1(t)-I_1)\cup \{m\}$. Case
(i) is established.

Suppose that $2\leq i\leq \lfloor |\mu|/p\rfloor$. Note that
$(g_it)_{11}=1$. Let $\sigma\in C_1(g_it)$ such that $\sigma(g_it)$
has numbers increasing down the first column. So
$\sgn(\sigma)e_{g_it}$ is a standard polytabloid. Since
$g_i(R_1(t)\cap I_i)=R_1(g_it)\cap I_i$, this gives the equivalent
statement. If $I_i\subseteq R_1(t)$, clearly $g_ie_t=e_t$. In the
case $I_i\subseteq C_1(t)$, we have $g_ie_t=\sgn(g_i)e_t=e_t$.
\end{proof}

Recall that for any point $\alpha=(\alpha_1,\ldots,\alpha_s)\in
k^s$ we define the element $u_\alpha=1+\sum^s_{i=1}\alpha_i(g_i-1)$
in the group algebra $kE_s$.

\begin{lem}\label{lemma to hook multiple of p} Let $\mu=(a,1^b)$,
$a+b=dp+r$ with $0\leq r\leq p-1$. Let $1\leq s\leq d$, $\alpha\in
k^{s}$ be a generic point and $t$ be a standard $\mu$-tableau. Fix
an integer $1\leq m\leq p-1$. If $(u_\alpha-1)^me_t=0$, then for any
$2\leq j\leq s$ the set $I_j$ lies entirely inside either the first
column of $\mu$ or the first row of $\mu$.
\end{lem}
\begin{proof} Suppose that there is some $2\leq j\leq s$ such
that $I_j\not\subseteq R_1(t)$ or $I_j\not\subseteq C_1(t)$. By Lemma \ref{subhooks occur} (ii), $g_j$
permutes the set of $\mu$-standard polytabloids up to a sign. The
size of the orbit $\mathcal{O}(e_t)$ under the action of $g_j$ is
$p$, up to a sign. Note that $(u_\alpha-1)^m$ is a linear
combination of some products of not more than $m$ copies of $g_i$'s
with $1\leq i\leq s$ (may be repeated). Fix an $m$-string
$\beta=(\beta_1,\ldots,\beta_m)$ with each $\beta_i\in
\{0,1,\ldots,s\}$, we write $g_\beta=g_{\beta_1}g_{\beta_2}\ldots
g_{\beta_m}$ assuming that $g_{\beta_0}=1$. Note that
$\lambda({g_j}^mt)=\lambda(t)$. We claim that $g_\beta e_t$ does not involve
${g_j}^me_t$ up to a sign unless and only unless
$\beta=\mathbf{j}=(j,\ldots,j)$. Once we have proved this claim,
since $g_\mathbf{j}$ occurs precisely once with coefficient
${\alpha_j}^m$ in the expansion of $(u_\alpha-1)^m$ and the point
$\alpha$ is generic, we conclude that $(u_\alpha-1)^me_t\neq 0$.

For each $1\leq i\leq d$, let $R_{1,i}(s)=R_1(s)\cap I_i$ for a
$\mu$-tableau $s$. Let $h,l$ be the multiplicities of $g_1,g_j$
appearing in $g_\beta$ respectively. Let $e_s$ be a standard
polytabloid involved in $g_\beta e_t$ such that
$\lambda(s)=\lambda(t)$. By Lemma \ref{subhooks occur} (i) and (ii),
we have $R_1(s)-I_1=(g_\beta{g_1}^{-h})(R_1(t)-I_1)$ and
$R_{1,j}(s)={g_j}^lR_{1,j}(t)$. So
$R_{1,j}({g_j}^mt)=R_{1,j}(s)=R_{1,j}({g_j}^lt)$ if and only if
$l=m$, i.e., $\beta=\mathbf{j}$.
\end{proof}

\begin{thm}\label{generic jordan type of hook partition of case p a multiple}
Let $\mu=(a,1^b)$ with $a+b=dp$ and $b=sp+b_0$ with $0\leq b_0\leq
p-1$. If $b_0$ is even, then the stable generic Jordan type of
$S^\mu{\downarrow_{E_d}}$ is $\left (1^{\binom{d-1}{s}}\right )$; otherwise, it is
$\left ((p-1)^{\binom{d-1}{s}}\right )$.
\end{thm}
\begin{proof} We prove the result by induction on $b$. If $b=0$,
then $S^\mu{\downarrow_{E_d}}$ is the trivial $kE_d$-module. So it has
Jordan type $(1)$. Suppose that we know the stable generic Jordan type
of $S^\mu{\downarrow_{E_d}}$ for some $\mu=(a,1^b)$ with $b=sp+b_0$
and $0\leq b_0\leq p-1$. Let $\lambda=(a-1,1^{b+1})$. Consider two
cases, $b_0<p-1$ or $b_0=p-1$. In the case $b_0<p-1$, the
$p$-residue of $b+1$ is $b_0+1$. The $kE_d$-module
$$S^{(a-1,1^b)}{\uparrow^{\mathfrak{S}_{dp}}_{\mathfrak{S}_{dp-1}}}{\downarrow_{E_d}}$$
is generically free. On the other hand, using the Branching Theorem, the
module has a filtration with factors $S^\mu{\downarrow_{E_d}}$,
$S^{(a-1,2,1^{b-1})}{\downarrow_{E_d}}$ and
$S^{\lambda}{\downarrow_{E_d}}$. Since the $p$-residue of $b$ is
strictly less than $p-1$, the partition $(a-1,2,1^{b-1})$ has
non-empty $p$-core. So $S^{(a-1,2,1^{b-1})}{\downarrow_{E_d}}$ is a
direct summand of
$S^{(a-1,1^b)}{\uparrow^{\mathfrak{S}_{dp}}_{\mathfrak{S}_{dp-1}}}{\downarrow_{E_d}}$
and it is generically free. So the stable generic Jordan types of
$S^\mu{\downarrow_{E_d}}$ and $S^\lambda{\downarrow_{E_d}}$ are
complementary. In the case $b_0=p-1$, we have $b+1\equiv 0(\mod p)$.
By the Branching Theorem, $S^{(a,1^{b+1})}{\downarrow_{E_d}}$ has a
filtration $S^\mu{\downarrow_{E_d}}$ and
$S^\lambda{\downarrow_{E_d}}$ reading from the top. Let $\alpha\in
k^d$ be a generic point, we construct the short exact sequence
$$\xymatrix{0\ar[r]& S^\lambda{\downarrow_{\langle
u_\alpha\rangle}}\ar[r]^-f& S^{(a,1^{b+1})}{\downarrow_{\langle
u_\alpha\rangle}}\ar[r]& S^\mu{\downarrow_{\langle
u_\alpha\rangle}}\ar[r]& 0}$$ where $f$ maps each standard
$\lambda$-polytabloid $e_t$ to the standard
$(a,1^{b+1})$-polytabloid $e_{\phi(t)}$ where $\phi(t)_{ij}=t_{ij}$
if $(i,j)\neq (1,a)$ and $\phi(t)_{1a}=dp+1$. Note that the set of
standard tableaux of $(a,1^{b+1})$ is the union of
$\Omega_1$ and $\Omega_2$ with $\Omega_1\cap \Omega_2=\varnothing$ and where $\Omega_1$ is the set consisting of
standard tableaux $s$ such that $s_{1a}=dp+1$ and $\Omega_2$ is the
set consisting of standard tableaux $s$ such that $s_{b+2,1}=dp+1$.
We claim that $f$ splits in the stable $k\langle u_\alpha\rangle$-module category. Once we have
done this, we have $S^{(a,1^{b+1})}{\downarrow_{\langle
u_\alpha\rangle}}\cong S^\mu{\downarrow_{\langle
u_\alpha\rangle}}\oplus S^\lambda{\downarrow_{\langle
u_\alpha\rangle}}$. Using Theorem \ref{generic jordan type of hook
partition of case p not a multiple} with $r=1$ and $b+1=(s+1)p$, the
stable generic Jordan type of $S^{(a,1^{b+1})}{\downarrow_{E_d}}$ is
$\left (1^{\binom{d}{s+1}}\right )$. By induction hypothesis, the stable generic
Jordan type of $S^\mu{\downarrow_{E_d}}$ is $\left (1^{\binom{d-1}{s}}\right )$. So
the stable generic Jordan type of $S^\lambda{\downarrow_{E_d}}$ is
$$\left (1^{\binom{d}{s+1}-\binom{d-1}{s}}\right )=\left (1^{\binom{d-1}{s+1}}\right ).$$ This
completes the inductive step.

We want to define a map $g:S^{(a,1^{b+1})}{\downarrow_{\langle
u_\alpha\rangle}}\to S^{\lambda}{\downarrow_{\langle
u_\alpha\rangle}}$ in the stable module category such that
$gf=\id_{S^\lambda{\downarrow_{\langle u_\alpha\rangle}}}$. For any
$s\in\Omega_1$, we define $g(e_s)=e_t$ where $t$ is the unique
standard $\mu$-tableau such that $\phi(t)=s$. Let $s\in\Omega_2$. If
$(u_\alpha-1)^{p-1}e_s=0$, by Lemma \ref{lemma to hook multiple of
p} with $s=d$ (we have abused the notation $s$), for each $2\leq j\leq d$ we have $I_j\subseteq
R_1(s)$ or $I_j\subseteq C_1(s)$. Since $b_0=p-1$ and
$s_{b+2,1}=dp+1$, it follows that $I_1\subseteq C_1(s)$. So
$(u_\alpha-1)e_s=0$ if $(u_\alpha-1)^{p-1}e_s=0$. Since $e_s$ is
fixed by $u_\alpha$, we may define $g(e_s)=0$. If
$(u_\alpha-1)^{p-1}e_s\neq 0$, we claim that the set $\{e_s,
(u_\alpha-1)e_s,\ldots, (u_\alpha-1)^{p-1}e_s\}$ is $k$-linearly
independent. Suppose that we have a relation
$$a_0e_s+a_1(u_\alpha-1)e_s+\ldots+a_{p-1}(u_\alpha-1)^{p-1}e_s=0$$
with $a_0,\ldots,a_{p-1}\in k$. Note that $(u_\alpha-1)^p=0$.
Multiplying the equation by $(u_\alpha-1)^{p-1}$, we get
$a_0(u_\alpha-1)^{p-1}e_s=0$. Since $(u_\alpha-1)^{p-1}e_s\neq 0$,
we have $a_0=0$. Inductively, by multiplying $(u_\alpha-1)^i$ for
some suitable $i$ to the equation, we show that
$a_0=a_1=\ldots=a_{p-1}=0$. So $e_s$ lies inside a free summand of
$S^{(a,1^{b+1})}{\downarrow_{\langle u_\alpha\rangle}}$. In this
case, we may define $g(e_s)=0$. The map $g$ gives a splitting for $f$
in the stable $k\langle u_\alpha\rangle$-module category. The proof is complete.
\end{proof}

Combining Corollary \ref{corollary to hook not multiple of p} and
Theorem \ref{generic jordan type of hook partition of case p a
multiple}, we have our main result Theorem \ref{complexity
corresponding to hook partitions}.

\thankyou{I would like to thank my supervisor Dave Benson for
valuable discussions and Johannes Orlob for pointing out the obvious generalization of 3.2.2 \cite
{DHDN} which leads to Theorem \ref{sgjt for signed permutation module} (ii).


\begin{thebibliography}{999}
\addcontentsline{toc}{section}{Bibliography}

\medskip
\bibitem{JALE} Jonathan L. Alperin and Leonard Evens, {\em Representations,
resolutions, and Quillen's dimension theorem}, J. Pure Appl. Algebra 22 (1981), 1-9.

\bibitem{GALS} George S. Avrunin and Leonard L. Scott,
{\em Quillen stractification for modules}, Invent. Math. 66(1982),
277-286.

\bibitem{DB} David J. Benson,
 {\em Representations and Cohomology I, II}, Cambridge
Univ. Press 30, 31.

\bibitem{JC} Jon F. Carlson,
{\em The complexity and varieties of modules}, Integral
representations and their applications, Oberwolfach 1980. Lecture
Notes in Math., vol.882, Springer-Verlag, Berlin/New York 1981,
415-422.

\bibitem{JC1} Jon F. Carlson, {\em The varieties and the cohomology ring of a module}, J. Algebra 85 (1983), 104-143.

\bibitem{SD2} Stephen Donkin, {\em Symmetric and exterior powers, linear source modules and representations of Schur superalgebras}, Proc. London Math. Soc. (3) 83 (2001) 647-680.

\bibitem{LE} Leonard Evens, {\em The Cohomology of Groups}, Oxford Science Publications, 1991.

\bibitem{EFJPAS} Eric M. Friedlander, Julia Pevtsova and Andrei Suslin,
{\em Generic and maximal Jordan types}, Invent. Math.
168(2007), 485-522.

\bibitem{GJ} Gordon D. James,
{\em The Representation Theory of the Symmetric Groups}, Lecture
Notes in Math., vol.682, Springer-Verlag, 1978.

\bibitem{GJAK} Gordon D. James and Adalbert Kerber,
{\em The Representation Theory of the Symmetric Groups},
Encyclopedia of Math. App., vol.16,
Addison-Wesley Publishing Company, 1981.

\bibitem{DH2} David J. Hemmer, {\em The complexity of certain Specht modules for the symmetric
group}, J. Algebraic Combin., inj press.

\bibitem{DHDN} David J. Hemmer and Daniel K. Nakano,
{\em Support varieties for modules over symmetric groups}, J. Algebra, 254(2002), 422-440.

\bibitem{KJL} Kay Jin Lim, {\em The varieties for some Specht modules}, J. Algebra 321(2009), 2287-2301.

\bibitem{WW} Wayne W. Wheeler, {\em Generic Module Theory}, J. Algebra , 185(1996), 205-228.

\bibitem{MW} Mark Wildon, {\em Two theorems on the vertices of
Specht modules}, Arch. Math. 81 (2003), 505-511.
\end{thebibliography}
\end{document}